\documentclass[11pt]{amsart}
\usepackage{amssymb}
\usepackage{amsmath}
\usepackage{amsfonts}
\usepackage{graphicx}

\usepackage{hyperref}
    \usepackage{aeguill}
    \usepackage{type1cm}

\theoremstyle{plain}
\newtheorem{thm}{Theorem}[section]

\newtheorem{claim}{Claim}

\newtheorem{corollary}[thm]{Corollary}

\newtheorem{lemma}[thm]{Lemma}

\newtheorem{proposition}[thm]{Proposition}

\newtheorem{theorem}[thm]{Theorem}

\newtheorem{op}[thm]{Open Problem}

\numberwithin{equation}{section}
\newcommand{\N}{\mathbb{N}}

\newcommand{\R}{\mathbb{R}}

\begin{document}

\title[Global approximation of convex functions on Banach spaces]{Global approximation of convex functions by differentiable convex functions on Banach spaces}
\author{Daniel Azagra}
\address{ICMAT (CSIC-UAM-UC3-UCM), Departamento de An{\'a}lisis Matem{\'a}tico,
Facultad Ciencias Matem{\'a}ticas, Universidad Complutense, 28040, Madrid, Spain }
\email{azagra@mat.ucm.es}

\author{Carlos Mudarra}
\address{ICMAT (CSIC-UAM-UC3-UCM), Calle Nicol\'as Cabrera 13-15, Campus de Cantoblanco, 28049 Madrid, Spain}
\email{carlos.mudarra@icmat.es}

\date{October 29, 2014}


\thanks{C. Mudarra was supported by Programa Internacional de Doctorado de la Fundaci\'on La Caixa--Severo Ochoa. Both authors were partially supported by Spanish Ministry Grant MTM2012-3431}


\begin{abstract}
We show that if $X$ is a Banach space whose dual $X^{*}$ has an equivalent locally uniformly rotund (LUR) norm, then for every open convex $U\subseteq X$, for every $\varepsilon >0$, and for every continuous and convex function $f:U \rightarrow \R$ (not necessarily bounded on bounded sets) there exists a convex function $g:X \rightarrow \R$ of class $C^1(U)$ such that $f-\varepsilon\leq g\leq f$ on $U.$ We also show how the problem of global approximation of {\em continuous} (not necessarily bounded on bounded sets) and convex functions by $C^k$ smooth convex functions can be reduced to the problem of global approximation of {\em Lipschitz} convex functions by $C^k$ smooth convex functions.
\end{abstract}

\maketitle

\section{Introduction and main results}

It is no doubt useful to be able to approximate convex functions by smooth convex functions. In $\R^{n}$, standard techniques (integral convolutions with mollifiers) enable us to approximate a given convex function by $C^{\infty}$ convex functions, uniformly on compact sets. If a given convex function $f:U\subseteq\R^{n}\to\R$ is not Lipschitz and one desires to approximate $f$ by smooth convex functions uniformly on the domain $U$ of $f$, then one has to work harder as, in absence of {\em strong convexity} of $f$, partitions of unity cannot be used to glue local approximations into a global one without destroying convexity. In a recent paper \cite{A} the first-named author devised a gluing procedure that permits to show that global approximation of (not necessarily Lipschitz of strongly) convex functions by smooth (or even real analytic) convex functions is indeed feasible.
We also refer to \cite{GreeneWu} and \cite{Smith} for information about this problem in the setting of finite-dimensional Riemannian manifolds, and to \cite{AF1, AF2} for the case of infinite-dimensional Riemannian manifolds.

In this note we will consider the question whether or not global approximation of continuous convex functions can be performed in Banach spaces. Let us briefly review the main techniques available in this setting for approximating convex functions by smooth convex functions. On the one hand, there are very fine results of Deville, Fonf, H\'ajek and Talponen \cite{DFH1, DFH2, HT} showing that if $X$ is the Hilbert space (or more generally a separable Banach space with a $C^m$ equivalent norm) then every bounded closed convex body in $X$ can be approximated by real-analytic (resp. $C^m$ smooth) convex bodies.
Via the implicit function theorem this yields that for every convex function $f:X\to\R$ which is bounded on bounded sets, for every $\varepsilon>0$, and for every bounded set $B\subset X$, there exists a $C^m$ smooth convex function $g:B\to\R$ such that $|f-g|\leq\varepsilon $ on $B$. Unfortunately these approximations $g$ are only defined on a bounded subset of $X$, so they cannot be used along with \cite[Theorem 1.2]{A} to solve the global approximation problem we are concerned with.

On the other hand, if $f:X\to\R$ is convex and Lipschitz and the dual space $X^{*}$ is LUR (we refer the reader to \cite{DGZ, FabianEtAl} for any unexplained terms in Banach space theory), then it is well known that the Moreau-Yosida regularizations $f_{\lambda}(x)=\inf_{y\in X}\{f(y)+\frac{1}{2\lambda}\|x-y\|^{2}\}$ are $C^1$ smooth and convex, and approximate $f$ uniformly on $X$ as $\lambda\to 0^{+}$ (see the proof of Theorem \ref{primer teorema} below). If $f$ is not Lipschitz but it is bounded on bounded subsets of $X$, then the $f_{\lambda}$ approximate $f$ uniformly on bounded subsets of $X$. And, if $f$ is only continuous, then the convergence of the $f_{\lambda}$ to $f$ is uniform only on compact subsets of $X$. By combining the gluing technique of \cite[Theorem 1.2]{A} with these results, one can deduce that every convex function which is bounded on bounded subsets of $X$ can be approximated by $C^1$ convex functions, uniformly on $X$; see \cite[Corollary 1.5]{A}.

However, as shown in \cite[Theorem 2.2]{BorFitzVan} or \cite[Theorem 8.2.2]{BorVan}, for every infinite-dimensional Banach space $X$ there exist continuous convex functions defined on all of $X$ which are not bounded on bounded sets of $X$. There are plenty of such examples, and they can be taken to be either smooth or nonsmooth. For instance, if $X=\ell_2$, the function $f(x)=\sum_{n=1}^{\infty}|x_{n}|^{2n}$ is real-analytic on $X$, but is not bounded on the ball of center $0$ and radius $2$ in $X$. On the other hand, if $\varphi:[0,\infty)\to [0, \infty)$ is a convex function such that $t\leq\varphi(t)\leq 2t$ and $\varphi$ is not differentiable at any rational number, then it is not difficult to see that the function $g(x)=\sum_{n=1}^{\infty}\varphi(|x_n|)^{2n}$ is continuous and convex on $X$, is not bounded on $B(0,2)$, and the set $\{x\in X \, : \, g \textrm{ is not differentiable at } x\}$ is dense.

In view of these remarks, even in the case when $X$ is the separable Hilbert space, the following result is new.

\begin{theorem}\label{primer teorema}
Let $U$ be an open convex subset of a Banach space $X$ such that $X^{*}$ has an equivalent LUR norm. Then, for every $\varepsilon >0$ and every continuous and convex function $f:U \rightarrow \R,$ there exists a convex function $g:U \rightarrow \R$ of class $C^1(U)$ such that $f-\varepsilon\leq g\leq f$ on $U.$
\end{theorem}

This will be proved by combining the above mentioned result on the Moreau-Yosida regularization of a convex function with the following refinement of \cite[Theorem 1.2]{A} which tells us that, in general, the problem of global approximation of {\em continuous} convex functions by $C^m$ smooth convex functions can be reduced to the problem of global approximation of {\em Lipschitz} convex functions by $C^m$ smooth convex functions.

\begin{theorem}\label{segundo teorema}
Let $X$ be a Banach space with the following property: every Lipschitz convex function on $X$ can be approximated by convex functions of class $C^m$, uniformly on $X.$ Then, for every $U\subseteq X$ open and convex, every continuous convex function on $U$ can be approximated by $C^m$ convex functions, uniformly on $U$.
\end{theorem}

From Theorem \ref{primer teorema} we will also deduce the following characterization of the class of separable Banach spaces for which the problem of global approximation of continuous convex functions by $C^1$ convex functions has a positive solution.

\begin{corollary}\label{corolario}
For a separable Banach space $X,$ the following statements are equivalent:
\begin{itemize}
\item[(i)] $X^*$ is separable.
\item[(ii)] For every $U\subseteq X$ open and convex, every continuous convex function $f:U \rightarrow \R$ and every $\varepsilon >0,$ there exists $g: X \rightarrow \R$ of class $C^1(U)$ and convex such that $f-\varepsilon \leq g \leq f$ on $U.$
\end{itemize}
\end{corollary}

In order to know whether or not similar results are true for higher order smoothness classes, and in view of Theorem \ref{segundo teorema} above, one would only need to solve the following problem.

\begin{op}
Let $X$ be a Hilbert space (or in general a Banach space possessing an equivalent  norm of class $C^m$), $f:X \rightarrow \R$ a Lipschitz and convex function, and $\varepsilon>0.$ Does there exist $\varphi : X \rightarrow \R$ of class $C^\infty$ (resp. $C^m$) and convex such that $|f-\varphi| \leq \varepsilon$ on $X$?
\end{op}

As a matter of fact, by combining Theorem \ref{segundo teorema} and the proof of \cite[Theorem 1.2]{A}, it would also be enough to solve the following.

\begin{op}
Let $X$ be a Hilbert space (or in general a Banach space possessing an equivalent  norm of class $C^m$), $f:X \rightarrow \R$ a Lipschitz and convex function, $B$ a bounded convex subset of $X$, and $\varepsilon>0.$ Does there exist $g : X \rightarrow \R$ of class $C^\infty$ (resp. $C^m$) and convex such that $g\leq f$ on $X$, and
 $f-\varepsilon\leq g$ on $B$?
\end{op}

\section{Proofs}

Let us first recall a couple of tools from \cite{A}.
\begin{lemma}[Smooth maxima]\label{smooth maxima}
For every $\varepsilon>0$ there exists a $C^\infty$ function $M_{\varepsilon}:\R^{2}\to\R$ with the following properties:
\begin{enumerate}
\item $M_{\varepsilon}$ is convex;
\item $\max\{x, y\}\leq M_{\varepsilon}(x, y)\leq \max\{x,y\}+\frac{\varepsilon}{2}$ for all $(x,y)\in\R^2$.
\item $M_{\varepsilon}(x,y)=\max\{x,y\}$ whenever $|x-y|\geq\varepsilon$.
\item $M_{\varepsilon}(x,y)=M_{\varepsilon}(y,x)$.
\item $\textrm{Lip}(M_{\varepsilon})=1$ with respect to the norm $\|\cdot\|_{\infty}$ in $\R^2$.
\item $y-\varepsilon\leq x<x'\implies M_{\varepsilon}(x,y)<M_{\varepsilon}(x',y)$.
\item $x-\varepsilon\leq y<y'\implies M_{\varepsilon}(x,y)<M_{\varepsilon}(x,y')$.
\item $x\leq x', y\leq y' \implies M_{\varepsilon}(x,y)\leq M_{\varepsilon}(x', y')$, with a strict inequality in the case when both $x<x'$ and $y<y'$.
\end{enumerate}
\end{lemma}
We call $M_{\varepsilon}$ a smooth maximum. In order to prove this Lemma, one first constructs a $C^{\infty}$ function $\theta:\R\to (0,
\infty)$ such that:
\begin{enumerate}
\item $\theta(t)=|t|$ if and only if $|t|\geq\varepsilon$;
\item $\theta$ is convex and symmetric;
\item $\textrm{Lip}(\theta)=1$,
\end{enumerate}
and then one puts
$$
M_{\varepsilon}(x,y)=\frac{x+y+\theta(x-y)}{2}.
$$
See \cite[Lemma 2.1]{A} for details.

Let us also restate Proposition 2.2 from \cite{A}.
\begin{proposition}\label{properties of M(f,g)}
Let $U$ be an open convex subset of $X$,
$M_{\varepsilon}$ as in the preceding Lemma, and let $f, g: U\to\R$
be convex functions. For every $\varepsilon>0$, the function
$M_{\varepsilon}(f,g):U\to\R$ has the following properties:
\begin{enumerate}
\item $M_{\varepsilon}(f,g)$ is convex.
\item If $f$ is $C^m$ on $\{x: f(x)\geq g(x)-\varepsilon\}$ and $g$ is $C^m$ on $\{x: g(x)\geq f(x)-\varepsilon\}$ then $M_{\varepsilon}(f,g)$ is $C^m$ on $U$. In particular, if $f, g$ are $C^m$, then so is $M_{\varepsilon}(f,g)$.
\item $M_{\varepsilon}(f,g)=f$ if $f\geq g+\varepsilon$.
\item $M_{\varepsilon}(f,g)=g$ if $g\geq f+\varepsilon$.
\item $\max\{f,g\}\leq M_{\varepsilon}(f,g)\leq \max\{f,g\} + \varepsilon/2$.
\item $M_{\varepsilon}(f,g)=M_{\varepsilon}(g, f)$.
\item $\textrm{Lip}(M_{\varepsilon}(f,g)_{|_B})\leq \max\{ \textrm{Lip}(f_{|_B}), \textrm{Lip}(g_{|_B}) \}$ for every ball $B\subset U$ (in particular $M_{\varepsilon}(f,g)$ preserves common local Lipschitz constants of $f$ and $g$).
\item If $f, g$ are strictly convex on a set $B\subseteq U$, then so is $M_{\varepsilon}(f,g)$.
\item If $f_1\leq f_2$ and $g_1\leq g_2$ then $M_{\varepsilon}(f_1, g_1)\leq M_{\varepsilon}(f_2, g_2)$.
\end{enumerate}
\end{proposition}

We are now ready to prove our results.

\begin{proof}[Proof of Theorem \ref{segundo teorema}]
Given a continuous convex function $f:U \rightarrow \R$ and $\varepsilon >0,$ we define, for each $n\in\N$,
$$E_n=\{x\in U \, | \,  f \textrm{ is $n$-Lipschitz on an open neighborhood of $x$}\}.$$
It is obvious that $E_n$ is an open subset of $U$ and $E_n \subseteq E_{n+1}.$ Since $f$ is continuous and convex, $f$ is locally Lipschitz and then, for every point $x\in U,$ there is an open set $x\in U_x\subset U$ and a positive integer $n$ for which $f$ is $n$-Lipschitz on $U_x.$ This proves that $U=\bigcup_{n=1}^\infty E_n.$ Now we set
 $$
\begin{array}{rccc}
f_n : X & \longrightarrow & \R  \\
   x & \longmapsto & \inf_{y \in U} \lbrace f(y)+ n\|x-y\| \rbrace,
\end{array} \quad n=1,2,\ldots
$$
\begin{claim}\label{primer claim}
For every $n \geq 1,$ the function $f_n$ has the following properties:
\begin{itemize}
\item[(i)] $f_n \leq f$ on $U.$
\item[(ii)] $f_n$ is $n$-Lipschitz on $X.$
\item[(iii)] $f_n$ is convex on $X.$
\item[(iv)] $f=f_n$ on $E_n.$
\end{itemize}
\end{claim}
The first three statements are well-known facts about infimal convolution on Banach spaces, see \cite{Stromberg} for a survey paper on these topics. In order to prove (iv), we only need to check that $f\leq f_n$ on $E_n.$ Let $x$ be a point of $E_n$ and let $U_x$ be an open subset of $U$ containing $x$ for which $f$ is $n$-Lipschitz on $U_x.$ Then the function $h_x: U \rightarrow \R$ given by $h_x(y)=f(y)+n||x-y||-f(x)$ for all $y\in U,$ has a local minimum at the point $x,$ where $h_x(x)=0.$ Since $h_x$ is clearly convex, this local minimum is in fact a global one, and therefore
$
f(x)\leq f(y)+n\|x-y\| \textrm{ for all } y\in U.
$
This proves (iv) of the Claim.

\bigskip

Since $f_n$ is Lipschitz, by assumption, for each $n\in\N$ we can find a function $h_n$ of class $C^m(X)$ and convex such that
\begin{equation}
f_n - \sum_{j=0}^{n-1} \frac{\varepsilon}{2^j} \leq h_n \leq f_n - \sum_{j=0}^{n-2} \frac{\varepsilon}{2^j}-\frac{\varepsilon}{2^n} \quad \text{on} \: X.
\end{equation}
Note that by Claim \ref{primer claim}, the last inequalities imply that
\begin{equation}\label{eq2}
 f- \sum_{j=0}^{n-1} \frac{\varepsilon}{2^j} \leq h_n \quad \text{on} \: E_n, \quad
h_n \leq f - \sum_{j=0}^{n-2} \frac{\varepsilon}{2^j}-\frac{\varepsilon}{2^n} \quad \text{on} \: U.
\end{equation}
Now, using the smooth maxima, we define a sequence $\lbrace g_n \rbrace_{n}$ of functions inductively setting $g_1=h_1$ and $g_n=M_{\varepsilon/10^n}(g_{n-1},h_n),$ for all $ n \geq 2$. According to the preceding Proposition, we have that $g_n$ is convex and of class $C^m$ on $X.$ We also know that
\begin{equation}\label{eq3}
\max\lbrace g_{n-1},h_n \rbrace \leq g_n \leq \max \lbrace g_{n-1},h_n\rbrace + \frac{\varepsilon}{10^n} \quad \text{on} \: X
\end{equation}
and $g_n(x)=\max \lbrace g_{n-1}(x),h_n(x) \rbrace$ at those points $x\in X$ for which $|g_{n-1}(x)-h_n(x)| \geq \varepsilon/10^n.$ The sequence $\lbrace g_n \rbrace_n$ satisfies the following properties:
\begin{claim}
For every $n \geq 2,$ we have
\begin{itemize} \label{segundo claim}
\item[(i)] $g_n=g_{n-1}$ on $E_{n-1}.$
\item[(ii)] $f-\varepsilon-\frac{\varepsilon}{2}- \cdots-\frac{\varepsilon}{2^{n-1}} \leq g_n$ on $E_n.$
\item[(iii)] $g_n \leq f-\frac{\varepsilon}{2}+\frac{\varepsilon}{10^2}+\cdots+\frac{\varepsilon}{10^n}$ on $U.$
\end{itemize}
\end{claim}
Property $(ii)$ is an obvious consequence of inequalities (\ref{eq2}) and (\ref{eq3}). Property $(i)$ can be proved as follows: given $x\in E_{n-1},$ the inequalities of (\ref{eq2}), together with the fact that $f=f_n=f_{n-1}$ on $E_{n-1}$, show that
$$
g_{n-1}(x) \geq h_{n-1}(x) \geq f_n(x)-\sum_{j=0}^{n-2} \frac{\varepsilon}{2^j} \geq h_n(x) +\frac{\varepsilon}{2^n} \geq h_n(x)+\frac{\varepsilon}{10^n},
$$
and this implies that $g_n(x)=g_{n-1}(x).$ We next show $(iii)$ by induction. In the case $n=2,$ our functions satisfy
$$
f_1-\varepsilon \leq h_1=g_1 \leq f_1-\frac{\varepsilon}{2}, \quad f_2-\varepsilon-\frac{\varepsilon}{2} \leq h_2 \leq f_2-\varepsilon-\frac{\varepsilon}{4}
$$
and $g_2=M_{\varepsilon/10^2}(g_1,h_2)$ on $X.$ It is enough to consider the following chain of inequalities: for all $x\in U$,
\begin{align*}
g_2(x) & \leq \max \lbrace h_2(x),g_1(x) \rbrace + \frac{\varepsilon}{10^2} \leq \max \lbrace f_2(x)-\varepsilon-\frac{\varepsilon}{4}, f_1(x)-\frac{\varepsilon}{2} \rbrace + \frac{\varepsilon}{10^2} \\
& \leq \max \lbrace f(x)-\varepsilon-\frac{\varepsilon}{4}, f(x)-\frac{\varepsilon}{2} \rbrace + \frac{\varepsilon}{10^2} \leq f(x)-\frac{\varepsilon}{2}+\frac{\varepsilon}{10^2}.
\end{align*}
Now we assume that for an integer $n \geq 2$ we have (iii), and we check that the same holds for $n+1.$ Let $x\in U.$ By combining the following two inequalities
$$
g_n(x) \leq f(x)-\frac{\varepsilon}{2}+\frac{\varepsilon}{10^2} +\cdots+\frac{\varepsilon}{10^n},
$$
$$
h_{n+1}(x) \leq f(x)-\varepsilon -\frac{\varepsilon}{2}-\frac{\varepsilon}{2^2}-\cdots-\frac{\varepsilon}{2^{n-1}} -\frac{\varepsilon}{2^{n+1}}
$$
(the first one is part of our induction hypothesis, and the second one is (\ref{eq2}) with $n+1$ in place of $n$),
and using that
$$
g_{n+1}=M_{\varepsilon/10^{n+1}}(g_n,h_{n+1}) \leq \max\lbrace h_{n+1},g_n \rbrace+\frac{\varepsilon}{10^{n+1}},
$$
we obtain
$$
g_{n+1}(x) \leq f(x)-\frac{\varepsilon}{2}+\frac{\varepsilon}{10^2}+\cdots+\frac{\varepsilon}{10^{n+1}}.
$$ This completes the proof of Claim \ref{segundo claim}.

\bigskip

Finally, we define
$$
g(x)=\lim_{n \to \infty} g_n(x) ,\quad \text{for all} \quad x\in U.
$$
Since we have that $g_{n+k}=g_n$ on each $E_n$ for all $k \geq 1,$ it is clear that $g$ is well defined and $g=g_n$ on $E_n$ for all $n$. And because each $E_n$ is an open subset of $U,$ this shows that $g\in C^m(U)$. Moreover the function $g$, being a limit of convex functions, is convex as well. To complete the proof of Theorem \ref{segundo teorema} let us see that $g$ is $2\varepsilon$-close to $f.$ Indeed, let $x\in U$ and take an integer $n\geq 2$ for which $x\in E_n.$ Using Claim \ref{segundo claim} and the preceding remarks about $g,$ we obtain
\begin{align*}
f(x)-2\varepsilon & \leq f(x)-\sum_{j=0}^{n-1} \frac{\varepsilon}{2^{j}} \leq g_n(x) = g(x) \\
& \leq f(x)-\frac{\varepsilon}{2}+\sum_{j=2}^{n} \frac{\varepsilon}{10^j} \leq f(x).
\end{align*}
Hence $f-2\varepsilon \leq g \leq f.$
\end{proof}

\begin{proof}[Proof of Theorem \ref{primer teorema}] Theorem \ref{primer teorema} actually is a corollary of \ref{segundo teorema}, because a Banach space $X$ whose dual $X^*$ is LUR has the approximation property mentioned in the hypotheses of Theorem \ref{segundo teorema}. This can be shown by using the infimal convolutions
$$
f_{\lambda}(x)=\inf_{y\in X}\{ f(y)+\frac{1}{2\lambda} \|x-y\|^{2}\},
$$
where $\|\cdot\|$ is an equivalent norm in $X$ whose dual norm is LUR. It is well known that if $f$ is convex and Lipschitz then $f_{\lambda}$ is $C^{1}$ smooth and convex, and converges to $f$ uniformly on $X$, as $\lambda\to 0^{+}$. For the smoothness part of this assertion, see \cite[Proposition 2.3]{van}. On the other hand, we next offer a proof of the fact that if $f$ is Lipschitz then $f_{\lambda}$ converges to $f$ uniformly on $X$ as $\lambda\to 0^{+}$. Observe first that in this case the infimum defining $f_{\lambda}(x)$ can be restricted to the ball $B\left(x, 2\lambda\textrm{Lip}(f)\right)$; indeed, if $d(x,y)>2\lambda\textrm{Lip}(f)$ then we have
$$f(y)+\frac{1}{2\lambda}d(x,y)^{2}\geq f(x)-\textrm{Lip}(f) d(x,y)+\frac{1}{2\lambda}d(x,y)^{2}\geq
f(x)\geq f_{\lambda}(x).$$ Now, one has
\begin{eqnarray*}
& & 0\leq f(x)-f_{\lambda}(x)= f(x)-\inf_{y\in B(x, 2\lambda\textrm{Lip}(f))}\{f(y)+\frac{1}{2\lambda}
d(x, y)^{2}\}\\
& & \leq\sup _{y\in B(x, 2\lambda\textrm{Lip}(f))}\{|f(x)-f(y)|+\frac{1}{2\lambda}
d(x, y)^{2}\} \\
& &\leq \textrm{Lip}(f) \left(2\lambda\textrm{Lip}(f)\right) +\frac{\left( 2\lambda\textrm{Lip}(f)\right)^{2}}{2\lambda},
\end{eqnarray*}
and the last term converges to $0$ as $\lambda\to 0^{+}$, so the assertion is proved.
\end{proof}

\begin{proof}[Proof of Corollary \ref{corolario}]
$(i)\implies (ii)$: If $X^*$ is separable, it is well known (see \cite[Theorem 8.6]{FabianEtAl} for instance) that there is an equivalent norm in $X$ whose dual norm is LUR on $X^*$, and therefore by using Theorem \ref{primer teorema} we obtain (ii).
\newline
$(ii) \implies (i)$: Take a convex function $\varphi \in C^1(X)$ such that
$
\|x\| -\frac{1}{4} \leq \varphi(x)\leq \|x\| \quad \text{for all} \quad x\in X.
$
It is easy to construct a function $h\in C^1(\R)$ such that $h(x)=1$ for all $x\leq 0$ and $h(x)=0$ for all $x \geq 3/4.$ Now, if we define the function $\psi:= h \circ \varphi$ it is obvious that $\psi $ is of class $C^1(X)$ with $\psi(0)=1.$ We also note that if $\|x\| \geq 1,$ then $\varphi(x) \geq 3/4$ and this implies that $\psi(x)=0.$ This shows that $\psi$ is a bump function of class $C^{1}(X)$. Because $X$ is separable, and according to \cite[Theorem 8.6]{FabianEtAl}, the dual space $X^*$ is separable too.
\end{proof}


\begin{thebibliography}{}


\bibitem{A}
D. Azagra, {\em Global and fine approximation of convex functions}. Proc. Lond. Math. Soc.
(3) 107 (2013), no. 4, 799--824.

\bibitem{AF1} D. Azagra and J. Ferrera, {\em Inf-convolution and regularization of convex functions  on Riemannian manifolds of nonpositive curvature}. Rev. Mat. Complut. 19 (2006), no. 2, 323--345.

\bibitem{AF2}
D. Azagra and J. Ferrera, {\em Regularization by sup--inf convolutions on Riemannian manifolds: An extension of Lasry--Lions theorem to manifolds of bounded curvature}. J. Math. Anal. Appl., in press. http://dx.doi.org/10.1016/j.jmaa.2014.10.022

\bibitem{BorVan}
J.M. Borwein, J.D. Vanderwerff, {\em Convex functions: constructions, characterizations and counterexamples}. Encyclopedia of Mathematics and its Applications, 109. Cambridge University Press, Cambridge, 2010.

\bibitem{BorFitzVan}
J.M. Borwein, S. Fitzpatrick, J.D. Vanderwerff, {\em Examples of convex functions and classifications of normed spaces}. J. Convex Anal. 1 (1994), no. 1, 61--73.

\bibitem{DGZ} R. Deville, G. Godefroy, V. Zizler, {\em Smoothness and renormings in Banach spaces}. Pitman Monographs and Surveys in Pure and Applied Mathematics, 64. Longman Scientific \& Technical, Harlow; copublished in the United States with John Wiley \& Sons, Inc., New York, 1993.

\bibitem{DFH1}
R. Deville, V. Fonf, P. H{\'a}jek, {\em Analytic and $C^k$
approximations of norms in separable Banach spaces}, Studia Math.
120 (1996), no. 1, 61--74.

\bibitem{DFH2}
R. Deville, V. Fonf, P. H{\'a}jek, {\em Analytic and polyhedral
approximation of convex bodies in separable polyhedral Banach
spaces}, Israel J. Math. 105 (1998), 139--154.

\bibitem{HT}
P. H\'ajek, J. Talponen, {\em Smooth approximations of norms in separable Banach spaces}. Quart. J. Math. 00 (2013), 1--13; doi:10.1093/qmath/hat053

\bibitem{FabianEtAl}
M. Fabian, P. Habala, P. H\'ajek, V. Montesinos, V. Zizler, {\em Banach space theory.
The basis for linear and nonlinear analysis.} CMS Books in Mathematics/Ouvrages de Math\'ematiques de la SMC. Springer, New York, 2011.

\bibitem{van} M. Fabian, P. H\'ajek, J. Vanderwerff, {\em On Smooth variational principles in Banach spaces}, J. Math. Anal. Appl. 197 (1996) no 1, 153-172.

\bibitem{GreeneWu}
R.E. Greene, and H. Wu, {\em $C\sp{\infty }$ convex functions and
manifolds of positive curvature}, Acta Math. 137 (1976), no. 3-4,
209--245.

\bibitem{Smith}
P.A.N. Smith, {\em Counterexamples to smoothing convex
functions}, Canad. Math. Bull. 29 (1986), no. 3, 308--313.

\bibitem{Stromberg}
T. Str\"omberg, {\em The operation of infimal convolution}. Dissertationes Math. (Rozprawy Mat.) 352 (1996), 58 pp.

\end{thebibliography}
\end{document}